\documentclass[11pt,twoside]{article}
\usepackage{psfrag}
\usepackage{epsfig}
\usepackage{graphicx}
\usepackage{amsmath}
\usepackage{amssymb}
\usepackage{amsthm}
\usepackage{subfigure}
\usepackage{cite}
\usepackage[latin1]{inputenc}

\newtheorem{lemma}{Lemma}
\newtheorem{theorem}{Theorem}

\begin{document}

\noindent
\textbf{\LARGE Local distributions for eigenfunctions and \\
for perfect colorings of q-ary hypercube
\footnote{This research is partially supported 
by the Russian Foundation for Basic Research under the grant no. 13-01-00463}
}
\date{}

\vspace*{5mm} \noindent
\textsc{Anastasia Vasil'eva} \hfill \texttt{vasilan@math.nsc.ru} \\
{\small Sobolev Institute of Mathematics,  \\
Novosibirsk State University, Novosibirsk, RUSSIA} \\

\medskip

\begin{center}
\parbox{11,8cm}{\footnotesize
\textbf{Abstract.} Under study are the eigenfunctions
and perfect colorings of the graph of $n$-dimensional $q$-ary
Hamming space. We obtain the
interdependence of local distributions of an eigenfunction in two
orthogonal faces. We prove also an analogous result for perfect
colorings. }
\end{center}

\baselineskip=0.9\normalbaselineskip

\section{Introduction}

We study the eigenfunctions and perfect colorings of $n$-dimensional $q$-ary hypercube.
The particular case of perfect colorings, which is extensively investigated now,
corresponds to the completely regular codes.
The aim of the paper is to provide a connection
between the local distributions in two orthogonal faces.
Earlier this question was considered in \cite{Vas99,Vas2009,H2012,Vas2013}
for the 1-error correcting perfect codes and perfect colorings in binary case ($q=2$).
In case $q>2$ the question is investigated in \cite{CHK} for the 1-error-correcting codes.
In \cite{Kro} a more general case of the direct product of graphs is studied;
however, the formula is not extended for the classes of graphs.

The paper is organised as follows:
In Section 2 we give some necessary notations and propositions.
In Section 3 we establish a formula for local weight enumerators
of an eigenfunction in a pair of orthogonal faces.
Using this, we obtain in Section 4 the formula for local weight enumerators
of a perfect coloring in a pair of orthogonal faces.
Both derived formulas are symmetric under choice of the face from the pair.

The results in the paper were published in part at the Seventh International 
Workshop on Optimal Codes and Related Topics, 
September, 2013 \cite{VasOC}.

\section{Preliminaries}

Consider the set ${\bf F}_q=\{0,1,\ldots,q-1\}$ as the group modulo $q$ and
${\bf F}_q^n$ as the abelian group ${\bf F}_q\times\ldots\times {\bf F}_q$.
We investigate functions and the colorings on the graph of ${\bf F}_q^n$
of \emph{$q$-ary $n$-dimensional hypercube}; in this graph two vertices are adjacent if
they differ in exactly one position.

Let $\alpha\in{\bf F}_q^n$ be an arbitrary vertex. Here and elsewhere $I$ denotes a subset of $\{1,\ldots, n\}$ and $\overline{I}=\{1,\ldots,n\}\backslash I$.
We denote the \emph{support} of the vertex $\alpha$  by $s(\alpha)$ (i.e., the set of its nonzero positions); the cardinality of the support is the \emph{Hamming weight} of $\alpha$ and is denoted by $wt(\alpha)$;
the \emph{Hamming distance} between two vertices $\alpha$ and $\beta$ that equals
the Hamming weight of $\alpha-\beta$ is denoted by $\rho(\alpha,\beta)$.
We write $W_i(\alpha)$ for the \emph{sphere} of radius $i$ centered at the vertex $\alpha$
(i.e., the set of all vertices with distance $i$ from $\alpha$) and we write
$B_i(\alpha)$  for the \emph{ball} of radius $i$ centered at the vertex $\alpha$ (i.e., the set of all vertices with distance at most $i$ from $\alpha$).
By definition, put
$$\Gamma_I(\alpha)=\{\beta\in{\bf F}_q^n \ : \ \beta_i=\alpha_i \ \forall \ i\notin I\} ,$$
then $\Gamma_I(\alpha)$ is an \emph{$|I|$-dimensional face},
it has the structure of ${\bf F}_q^{|I|}$. We write simply $W_i$ and $\Gamma_I$ instead of $W_i(\alpha)$ and $\Gamma_I(\alpha)$ in the case of the all-zero vertex $\alpha$.
Two faces $\Gamma_I(\alpha)$ and $\Gamma_J(\beta)$ are \emph{orthogonal} if $J=\overline{I}$.
Obviously, two orthogonal faces have exactly one common vertex.
Given $\alpha,\beta\in {\bf F}_q^n$,
we denote $\langle\alpha,\beta\rangle=\alpha_1\beta_1+\ldots+\alpha_n\beta_n\mod q .$

Let us consider the set of all functions $f: {\bf F}_q^n\longrightarrow \mathbb{C}$
as $q^n$-dimensional vector space $V$ over the complex field $\mathbb{C}$.
Let $\xi=e^{2\pi \sqrt{-1}/q}$.
For $\beta\in {\bf F}_q^n$, the~function $\varphi^{\beta}\in V$, where
$$\varphi^\beta(\alpha) = \xi^{\langle\alpha,\beta\rangle}, \ \ \ \ \alpha\in {\bf F}_q^n,$$
is called the \emph{character}.
All characters $\varphi^\beta, \ \beta\in {\bf F}_q^n$,
form the orthogonal basis of the vector space $V$ with respect
to the \emph{inner product} $\langle \ , \ \rangle$ defined as follows:
$$\langle f,g\rangle= \sum_{\beta\in {\bf F}_q^n} f(\beta) \overline{g(\beta)} .$$
The \emph{Fourier transform} $\widehat{f}$ of the function $f$ is defined as the inner product with the characters:
\begin{equation}\label{Ftran}
\widehat{f}(\alpha)=\langle f,\varphi^\alpha \rangle=
\sum_{\beta\in {\bf F}_q^n} f(\beta) \overline{\xi^{\langle\alpha,\beta\rangle}},
 \ \ \ \alpha\in {\bf F}_q^n.
\end{equation}
The initial function $f$ can be presented in the basis of the characters:
\begin{equation}\label{Ftran-inv}
f(\alpha)=q^{-n}\sum_{\beta\in {\bf F}_q^n} \widehat{f}(\beta) \xi^{\langle\alpha,\beta\rangle}, \ \ \ \alpha\in {\bf F}_q^n.
\end{equation}

\begin{lemma} \label{Lvec}
Let $\beta\in{\bf {\bf F}}_q^n$ and $I\subseteq\{1,\ldots,n\}$. Then
$$\sum_{\alpha\in\Gamma_I} \varphi^{\beta}(\alpha) x^{|I|-|s(\alpha)|} y^{|s(\alpha)|} =
 (x-y)^{|I\bigcap s(\beta)|}
(x+(q-1)y)^{|I|-|I\bigcap s(\beta)|} .$$
\end{lemma}
\begin{proof}
Let $|I|=k$. Without loss of generality assume that $I=\{1,\ldots,n\}$.
By definition of the characters,
$$\sum_{\alpha\in\Gamma_I} \varphi^{\beta}(\alpha) x^{|I|-|s(\alpha)|} y^{|s(\alpha)|} =
\sum_{\alpha_1=0}^{q-1} \ldots \sum_{\alpha_k=0}^{q-1}
\prod_{i=1}^k \xi^{\alpha_i \beta_i} x^{1-|s(\alpha_i)|}y^{|s(\alpha_i)|} .$$
(For $a\in{\bf F}_q$ it holds $|s(a)|=0$ if $a=0$ and $|s(a)|=1$ if $a\neq 0$.)
Then we change the order of summations and  multiplication:
\begin{equation}\label{inL1}
\prod_{i=1}^k \sum_{\alpha_i=0}^{q-1} \xi^{\alpha_i \beta_i} x^{1-|s(\alpha_i)|}y^{|s(\alpha_i)|} .
\end{equation}
Owing to the properties of the  primitive root of unity, we have
$$\sum_{a=0}^{q-1} \xi^{a b}=
\left\{\begin{array}{ll} 0 , & b\neq 0 , \\
q, & b=0 \end{array} , \right.$$
and therefore
$$\sum_{a=0}^{q-1} \xi^{a b} x^{1-|s(a)|} y^{|s(a)|}=
\left\{\begin{array}{ll} x-y , & b\neq 0 ,
\\ x+(q-1)y, & b=0 . \end{array} \right.$$
Applying this to (\ref{inL1}), we finally obtain
$$(1-t)^{|I\bigcap s(\beta)|}
(1+(q-1)t)^{|I|-|I\bigcap s(\beta)|} .$$
\end{proof}

Now we introduce the concept of a local distribution. By definition, put
$$v^{I,f}_j (\alpha)= \sum_{\beta\in\Gamma_I(\alpha)\bigcap W_j(\alpha)} f(\beta) ,$$
the vector $v^{I,f}(\alpha)=(v^{I,f}_0(\alpha),\ldots,v^{I,f}_{|I|}(\alpha))$
is called the \emph{local distribution of the function $f$
in the face $\Gamma_I(\alpha)$ with respect to the vertex $\alpha$}
or shortly the \emph{$(I,\alpha)$-local distribution} of $f$.  We say that the polynomial
$$g^{I,\alpha}_f(x,y) = \sum_{j=0}^{|I|} v^{I,f}_j(\alpha) y^j x^{|I|-j} =
\sum_{\beta\in\Gamma_I(\alpha)} f(\beta) y^{|s(\beta-\alpha)|}
x^{|I|-|s(\beta-\alpha)|}$$
is a \emph{local weight enumerator of the function $f$
in the face $\Gamma_I(\alpha)$ with respect to the vertex $\alpha$}
or shortly the \emph{$(I,\alpha)$-local weight enumerator} of $f$.
We omit $\alpha$ (in all notations) if $\alpha=(0,\ldots,0)$. 

Let us describe the local weight enumerator of an arbitrary function
in terms of its Fourier coefficients:
\begin{lemma} \label{gI}
Let $f$ be an arbitrary function. Then
\begin{equation} \label{enum_Four}
g^I_f(x,y) =  q^{-n}\sum_{\beta\in {\bf F}_q^n}\widehat{f}(\beta)
(x+(q-1)y)^{|I|-|I\bigcap s(\beta)|}(x-y)^{|I\bigcap s(\beta)|} .
\end{equation}
\end{lemma}
\begin{proof} By Lemma \ref{Lvec},
$$g^I_f(x,y) =
\sum_{\beta\in\Gamma_I} f(\beta) y^{|s(\beta)|} x^{|I|-|s(\beta)|} $$
$$= q^{-n}\sum_{\delta\in {\bf F}_q^n}\widehat{f}(\delta)
\sum_{\beta\in\Gamma_I} \xi^{\langle \beta,\delta \rangle} x^{|I|-|s(\beta)|} y^{|s(\beta)|} .$$
Then we can apply Lemma \ref{Lvec} and obtain (\ref{enum_Four}).
\end{proof}

\section{Eigenfunctions}
The first object of our consideration is the set of all eigenfunctions
of the $n$-dimensional $q$-ary hypercube ${\bf {\bf F}}_q^n$.
As usual, we refer to as the \emph{eigenvalue of a graph}
the eigenvalue of its adjacency matrix.
It is known that the eigenvalues $\lambda$ of the
graph of $n$-dimensional $q$-ary hypercube are equal to
$$\lambda_h=(q-1)n-qh, \ \ \ h=0,1,\ldots,n,$$
here $h$ is called the \emph{number of the eigenvalue} $\lambda_h$.
Obviously, an eigenvalue $\lambda$ has the number $h=h(\lambda)= \frac{(q-1)n-\lambda}{q}$.
The corresponding eigenfunctions (we call them \emph{$\lambda$-functions}) satisfy the equations
\begin{equation} \label{ball}  \sum_{\beta\in W_1(\alpha)} f(\beta) = \lambda_h f(\alpha), \ \ \ \alpha\in{\bf F}_q^n,
\end{equation}
or in the matrix form:
$$Df=\lambda_h f,$$
where $D$ is the adjacency matrix of ${\bf F}_q^n$
and $f$ is a vector of the function $f$ values.
It is easy to see that the Fourier coefficients $\widehat{f}(\alpha)$
of a $\lambda$-function $f$ equal zero apart from the case, where
the Hamming weight of $\alpha$ is equal to the number of $\lambda$.

We are going to derive the interdependence between the local weight enumerators
for an eigenfunction in two orthogonal faces.
\begin{theorem} \label{q-eig_loc}
Let $\lambda$ be an eigenvalue of ${\bf F}_q^n$ with the number $h=\frac{(q-1)n-\lambda}{q}$, \
let $f$ be an arbitrary $\lambda$-function, and let $\alpha\in{\bf F}_q^n$. Then
$$(x+(q-1)y)^{h-|\overline{I}|} g^{\overline{I},\alpha}_f(x,y) = (x'+(q-1)y')^{h-|I|} g^{I,\alpha}_f(x',y') ,$$
where $x'=x+(q-2)y, \ y'=-y $.
\end{theorem}
\begin{proof}
The faces $\Gamma_I(\alpha)$ and $\Gamma_{\overline{I}}(\alpha)$ are orthogonal.
Without loss of generality assume that $\alpha$ is the all-zero vertex.
Using Lemma \ref{gI}, we can express the $(\overline{I},{\bf 0})$-local weight enumerator
of the $\lambda$-function $f$ in terms of
the Fourier coefficients:
$$g_f^{\overline{I}}(x,y) = q^{-n}\sum_{\beta\in {\bf F}_q^n}\widehat{f}(\beta)
(x+(q-1)y)^{n-|I|-|s(\beta)|+|I\bigcap
s(\beta)|}(x-y)^{|s(\beta)|-|I\bigcap s(\beta)|} .$$
Since
$ \widehat{f}(\beta)=0$ for every $\beta\notin W_h,$
the summation can be taken over all vertices
of weight $h$ instead of all vertices of ${\bf F}_q^n$. This implies
$$g^{\overline{I}}_f(x,y) =
q^{-n}(x+(q-1)y)^{n-|I|-h}(x-y)^{h-|I|} \phantom{(x-y)^{|I|-|I\bigcap s(\beta)|} }$$
$$\phantom{(x-y)^{h-|I|}}\times \sum_{\beta\in W_h}\widehat{f}(\beta)
(x+(q-1)y)^{|I\bigcap s(\beta)|}(x-y)^{|I|-|I\bigcap s(\beta)|} .$$
We choose new variables $x'$ and $y'$ such that
$$\left\{ \begin{array}{rcl}
    x'+(q-1)y' & = & x-y, \\
    x'-y' & = & x+(q-1)y,
  \end{array} \right.   \ \ \ \mbox{or} \ \ \ \ \ \ \
\left\{ \begin{array}{rcl}
    x'&=&x+(q-2)y, \\
    y'&=&-y .
  \end{array} \right.   $$
Hence,
$$g^{\overline{I}}_f(x,y) =
q^{-n}(x+(q-1)y)^{n-|I|-h}(x-y)^{h-|I|} \phantom{(x-y)^{|I|-|I\bigcap s(\beta)|}} $$
$$\phantom{(x-y)^{h-|I|}}\times\sum_{\beta\in W_h}\widehat{f}(\beta)
(x'-y')^{|I\bigcap s(\beta)|}(x'+(q-1)y')^{|I|-|I\bigcap s(\beta)|} .$$
Comparing with Lemma \ref{gI}, we finally have
$$g^{\overline{I}}_f(x,y) = (x+(q-1)y)^{n-|I|-h}(x'+(q-1)y')^{h-|I|} g^I_f(x',y') .$$
\end{proof}

\section{Perfect colorings}
In this section we prove an analog of Theorem \ref{q-eig_loc} for perfect colorings.

The partition
$C=(C_1,\ldots,C_r)$ of ${\bf F}_q^n$ is called a \emph{perfect $r$-coloring}
(or an \emph{equitable partition}, or a \emph{partition design})
with the \emph{parameter matrix} $S=(s_{ij})_{i,j=1,\ldots,r}$ \
if for every $i,j\in\{1,\ldots,r\}$ and each vertex $\alpha\in C_i$
the number of vertices $\beta\in C_j$ at distance 1 from $\alpha$ is equal to $s_{ij}$.
Present a perfect $r$-coloring by $(0,1)$-matrix $C$ of size $q^n\times r$
with the rows corresponding to the vertices of ${\bf F}_q^n$ and
the columns corresponding to the colors $\{1,\ldots, r\}$.
The matrix $C$ is defined as follows:
each row has only one nonzero position that marks the color of the corresponding vertex.
In these terms the coloring is perfect if
\begin{equation}\label{DCCS}
DC=CS ,
\end{equation}
where $D$ is the adjacency matrix of the hypercube ${\bf F}_q^n$.

We define a local distribution of a coloring as a local distribution of characteristic functions of the colors. More precisely, a \emph{local distribution of the coloring $C$ in the face $\Gamma_I(\alpha)$ with respect to the vertex} $\alpha$ (or \emph{$(I,\alpha)$-local distribution})
is the $r\times (|I|+1)$-matrix
$$v^{I,C}(\alpha) = \left(\begin{array}{ccc}
                      v_0^{I,C_1}(\alpha) & \ldots & v_{|I|}^{I,C_1}(\alpha) \\
                      \vdots &  & \vdots \\
                      v_0^{I,C_r}(\alpha) & \ldots & v_{|I|}^{I,C_r}(\alpha)
                    \end{array}\right) ,$$
where $v_j^{I,C_i}(\alpha)= |C_i\bigcap W_j(\alpha)\bigcap\Gamma_I(\alpha)| , \ \ i=1,\ldots,r$,
and $j=0,\ldots,|I|$.
Let $g^{I,\alpha}_{C_i}(x,y), \ i=1\ldots,r,$ be the $(I,\alpha)$-local weight enumerator
of the $i$th color $C_i$; i.e.,
$$g^{I,\alpha}_{C_i}(x,y)  = \sum_{j=0}^{|I|} v^{I,C_i}_j(\alpha) y^j x^{|I|-j} .$$
The vector-function
$$g^{I,\alpha}_C(x,y) = (g^{I,\alpha}_{C_1}(x,y),\ldots,g^{I,\alpha}_{C_r}(x,y))$$
is called the \emph{local weight enumerator of the coloring $C$ in the face $\Gamma_I(\alpha)$ with respect to the vertex} $\alpha$ (or the \emph{$(I,\alpha)$-local weight enumerator}).

The next theorem is an analog of Theorem \ref{q-eig_loc} for perfect colorings.
\begin{theorem} \label{q-col_loc}
Let $C=(C_1,\ldots,C_r)$ be an arbitrary perfect coloring
of ${\bf F}_q^n$ with parameter matrix $S$ and $\alpha\in{\bf F}_q^n$.
Put $h(S)=\frac{(q-1)nE-S}{q}$, \ where $E$ is an identity matrix.
Then
\begin{equation}\label{Th-col}
g^{\overline{I},\alpha}_C(x,y) (x+(q-1)y)^{h(S)-|\overline{I}|E} =
g^{I,\alpha}_C(x',y') (x'+(q-1)y')^{h(S)-|I|E} .
\end{equation}
\end{theorem}
\begin{proof}
Without loss of generality assume that $\alpha=(0,\ldots,0)$.

Perfect colorings are closely related with eigenfunctions of the hypercube. Indeed, let
$\mu_1, \ldots,\mu_r$ be the all eigenvalues (not necessarily distinct)
of the parameter matrix $S$ and let $T^1,\ldots,T^r$ be the linearly independent
eigenvectors of~$S$ that corresponds to the eigenvalues; i.e.,
$$ST^i = \mu_iT^i, \ \ i=1,\ldots,r.$$
Thus, for the matrices $T=[T^1,\ldots,T^r]$ and
$M=diag\{\mu_1,\ldots,\mu_r\}$ it holds
$$ST=T M .$$
Multiplying both sides of (\ref{DCCS}) by $T$ and applying the last equation,
we have for the matrix
\begin{equation}\label{FCT}
    F=CT
\end{equation}
that
$$DF=DCT=CST = CTM=FM .$$
It means that the columns $F^1,\ldots,F^r$ of $F$
are the eigenfunctions of $D$ or $\lambda$-functions; i.e.,
$$DF^i=\mu_i F^i , \ \ \ i=1\ldots,r.$$
Applying Theorem \ref{q-eig_loc} to these $\lambda$-functions, we have
\begin{equation}\label{col-1}
(x+(q-1)y)^{h_i-|\overline{I}|} g^{\overline{I}}_{F^i}(x,y) = (x'+(q-1)y')^{h_i-|I|} g^I_{F^i}(x',y'), \ \ i=1,\ldots,r,
\end{equation}
where for $i=1\ldots,r$ the value $h_i$ is equal to the number
of the eigenvalue $\mu_i$ of the hypercube ${\bf F}_q^n$; i.e.,
$h_i= \frac{(q-1)n-\mu_i}{q}$.
Put $g_F=(g_{F^1},\ldots,g_{F^r}) $ and
$$M_I(x,y)=diag\left\{(x+(q-1)y)^{h_1-|I|}, \ldots,(x+(q-1)y)^{h_r-|I|}\right\} .$$
So we can rewrite the equations (\ref{col-1}) in terms of these matrices:
$$g_F^{\overline{I}}(x,y)M_{\overline{I}}(x,y)= g_F^I(x',y')M_I(x',y') .$$
It follows from (\ref{FCT}) that
$$g_F = (g_{F^1},\ldots,g_{F^r}) =
(g_{C^1},\ldots,g_{C^r}) T = g_C T.$$
Therefore, we obtain
\begin{equation}\label{gTM}
g_C^{\overline{I}}(x,y)T M_{\overline{I}}(x,y)= g_C^I(x',y')T M_I(x',y').
\end{equation}
Then we multiply both sides of (\ref{gTM}) by $T^{-1}$ and recall
the definition of a matrix function:
$$(x+(q-1)y)^{\frac{(q-1)nE-S}{q}-|I|E} = T M_I(x,y) T^{-1},$$
which gives (\ref{Th-col}) and concludes the proof.
\end{proof}

\end{document}